\documentclass[12pt]{amsart}

\newtheorem{theorem}{Theorem}[section]
\newtheorem{proposition}[theorem]{Proposition}
\newtheorem{lemma}[theorem]{Lemma}
\newtheorem{corollary}[theorem]{Corollary}

\theoremstyle{definition}

\newtheorem{example}[theorem]{Example}

\newtheorem{conjecture}[theorem]{Conjecture}
\newtheorem{remark}[theorem]{Remark}

\newcommand{\NN}{ \ensuremath{\mathbb{N}}}

\newcommand{\init}{\ensuremath{\mathrm{in}}\hspace{1pt}}
\newcommand{\lex}{{\mathrm{lex}}}

\newcommand{\reg}{\ensuremath{\mathrm{reg}}\hspace{1pt}}

\newcommand{\Tor}{\ensuremath{\mathrm{Tor}}\hspace{1pt}}

\newcommand{\aaa}{\mathbf{a}}
\newcommand{\bb}{\mathbf{b}}
\newcommand{\ee}{\mathbf{e}}

\def\cocoa{{\hbox{\rm C\kern-.13em o\kern-.07em C\kern-.13em o\kern-.15em A}}}

\newcommand{\mult}{\mathrm{mult}}
\newcommand{\lcm}{\mathrm{lcm}}

\begin{document}

\title{Regularity bounds for binomial edge ideals}

\author{Kazunori Matsuda}
\address{
Kazunori Matsuda,
Graduate School of Mathematics, 
Nagoya University, 
Furocho, Chikusaku, Nagoya 464-8602, Japan
}

\author{Satoshi Murai}
\address{
Satoshi Murai,
Department of Mathematical Science,
Faculty of Science,
Yamaguchi University,
1677-1 Yoshida, Yamaguchi 753-8512, Japan.
}

\dedicatory{Dedicated to Professor J\"urgen Herzog on the occasion of his 70th birthday}

\begin{abstract}
We show that the Castelnuovo--Mumford regularity of the binomial edge ideal of a graph
is bounded below by the length of its longest induced path and bounded above by the number of its vertices.
\end{abstract}

\maketitle

\section{Introduction}

Let $G$ be a simple graph on the vertex set $[n]=\{1,2,\dots,n\}$.
The {\em binomial edge ideal $J_G$} of $G$,
introduced by Herzog et.al.\ \cite{HHHKR} and Ohtani \cite{Oh},
is the ideal in 
the polynomial ring $S=K[x_1,\dots,x_n,y_1,\dots,y_n]$ over a field $K$,
defined by
$$J_G=(x_iy_j-x_jy_i: \{i,j\} \mbox{ is an edge of $G$}).$$
From an algebraic view point, it is of interest to study relations between algebraic properties of $J_G$
and combinatorial properties of $G$.
In this note, we prove the following simple combinatorial bounds for the regularity of binomial edge ideals.

\begin{theorem}
\label{1.1}
Let $G$ be a simple graph on $[n]$ and let $\ell$ be the length of the longest induced path of $G$.
Then
$$\ell + 1 \leq \reg (J_G) \leq n.$$
\end{theorem}



\section{A lower bound}
In this section, we prove a lower bound in Theorem \ref{1.1}.
Throughout the paper,
we will use the standard terminologies of graph theory in \cite{D}.

We consider the $\NN^n$-grading of $S$
defined by $\deg x_i = \deg y_i = \ee_i$,
where $\ee_i$ is the $i$-th unit vector of $\NN^n$.
Binomial edge ideals are $\NN^n$-graded by definition.
For an $\NN^n$-graded $S$-module $M$ and $\aaa \in \NN^n$,
we write $M_\aaa$ for the graded component of $M$ of degree $\aaa$
and write $\beta_{i,\aaa}(M)=\dim_K \Tor_i(M,K)_\aaa$ for the {\em $\NN^n$-graded Betti numbers} of $M$.
Also, for $\aaa=(a_1,\dots,a_n) \in \NN^n$,
let $\mathrm{supp}(\aaa)=\{i \in [n]: a_i \ne 0\}$ and $|\aaa|=a_1+ \cdots + a_n$.
Then the {\em $\NN$-graded Betti numbers} of $M$ are $\beta_{i,j}(M)=\sum_{\aaa \in \NN^n, |\aaa|=j} \beta_{i,\aaa}(M)$
and the {\em (Castelnuovo--Mumford) regularity} of $M$ is
$$\reg(M)= \max\{j: \beta_{i,i+j}(M) \ne 0 \mbox{ for some }i\}.$$

For a simple graph $G$ on the vertex set $[n]$
and for a subset $W \subset [n]$,
we write $G_W$ for the induced subgraph of $G$ on $W$.
For convenience we consider that $G_W$ has the vertex set $[n]$
and regard $J_{G_W}$ as an ideal of $S$.

\begin{lemma}
\label{2.1}
Let $G$ be a simple graph on $[n]$ and let $W \subset [n]$.
Then, for any $\aaa \in \NN^n$ with $\mathrm{supp}(\aaa) \subset W$,
one has
$$\beta_{i,\aaa}(J_G)=\beta_{i,\aaa}(J_{G_W}) \ \ \mbox{ for all }i.$$
\end{lemma}

\begin{proof}
Let
$$
\mathcal F: 0 \longrightarrow \bigoplus_{\aaa \in \NN^n} S^{\beta_{p,\aaa}(J_G)}(-\aaa)
\longrightarrow \cdots 
\longrightarrow \bigoplus_{\aaa \in \NN^n} S^{\beta_{0,\aaa}(J_G)}(-\aaa)
\stackrel {\phi} \longrightarrow S
$$
be the $\NN^n$-graded minimal free resolution of $S/J_G$,
where $p$ is the projective dimension of $J_G$.
Consider its subcomplex
$$
\mathcal F': 0 \longrightarrow \bigoplus_{\aaa \in \NN^n \atop \mathrm{supp}(\aaa) \subset W} S^{\beta_{p,\aaa}(J_G)}(-\aaa)
\longrightarrow \cdots
\longrightarrow \bigoplus_{\aaa \in \NN^n\atop \mathrm{supp}(\aaa) \subset W} S^{\beta_{0,\aaa}(J_G)}(-\aaa)
\stackrel {\phi'} \longrightarrow S.
$$
We claim that $\mathcal F'$ is the minimal free resolution of $S/J_{G_W}$.
It is clear that $\mathrm{coker}\ \phi' = S/J_{G_W}$.
Hence what we must prove is that $\mathcal F'$ is acyclic.
To prove this, it is enough to show that the multigraded component $\mathcal F'_\aaa$ is acyclic
for any $\aaa \in \NN^n$ with $\mathrm{supp}(\aaa) \subset W$.

Let $\aaa\in \NN^n$ with $\mathrm{supp} (\aaa) \subset W$.
Since, for any $\bb \in \NN^n$,
$S(-\bb)_\aaa$ is non-zero if and only if $\aaa - \bb$ is non-negative,
we have
$$\mathcal F_\aaa = \mathcal F'_\aaa,$$
which implies that $\mathcal F'_\aaa$ is acyclic since $\mathcal F$ is a minimal free resolution.
\end{proof}

\begin{corollary}
\label{2.2}
With the same notation as in Lemma \ref{2.1},
one has $\beta_{i,j}(J_G) \geq \beta_{i,j}(J_{G_W})$ for all $i,j$.
\end{corollary}

\begin{corollary}
\label{2.3}
Let $G$ be a simple graph on $[n]$ and let $\ell$ be the length of the longest induced path of $G$.
Then $\reg (J_G) \geq \ell +1$.
\end{corollary}

\begin{proof}
Observe that the binomial edge ideal of a path of length $\ell$
is a complete intersection having $\ell$ generators of degree $2$
and has the regularity $\ell+1$.
Then the statement follows from Corollary \ref{2.2}.
\end{proof}

\section{An upper bound}
In this section,  we prove an upper bound in Theorem \ref{1.1}.

We consider the $\NN^{2n}$-grading of $S$
defined by $\deg x_i = \ee_i$ and $\deg y_i= \ee_{i+n}$.
Binomial edge ideals are not $\NN^{2n}$-graded but monomial ideals in $S$ are $\NN^{2n}$-graded.
To simplify the notation,
we identify the multidegree $(\aaa,\bb) =(a_1,\dots,a_n,b_1,\dots,b_n)\in \NN^{2n}$
and a monomial $x^\aaa y^\bb=x_1^{a_1} \cdots x_n^{a_n} y_1^{b_1} \cdots y_n^{b_n}$,
and, for an $\NN^{2n}$-graded $S$-module $M$,
 write
$$\beta_{i,x^\aaa y^\bb}(M)= \beta_{i,(\aaa,\bb)}(M).$$
Also,
we write
$$P(M,t)= \sum_{k=0}^{2n} \sum_{(\aaa,\bb)\in \NN^{2n}} \beta_{k,(\aaa,\bb)}(M) x^\aaa y^\bb t^k$$
for the  ($\NN^{2n}$-graded) {\em Poincar\'e series} of $M$.

\begin{lemma}
\label{3.1}
Let $m_1,\dots,m_g$ be monomials in $S$ and $I=(m_1,\dots,m_g)$.
Then
$$
P(S/I,t) \leq 1 + \sum_{m_j \not \in (m_1,\dots,m_{j-1})}  P\big(S/\big( (m_1,\dots,m_{j-1}):m_j\big),t \big) m_j t,
$$
where the inequality is coefficient-wise.
\end{lemma}


\begin{proof}
The assertion follows from the short exact sequence
$$
0 \longrightarrow
S/\big((m_1,\dots,m_{j-1}):m_j \big) \stackrel{\times m_j} \longrightarrow
S/(m_1,\dots,m_{j-1}) \longrightarrow
S/(m_1,\dots,m_{j}) \longrightarrow 0
$$
for $j=2,3,\dots,g$,
by mapping cone construction (cf.\ \cite[Construction 27.3]{P}).
\end{proof}

We now consider binomial edge ideals.
In the rest of this section,
we fix a simple graph $G$ on $[n]$.
We say that a path
$$P: s =v_0 \to v_1 \to \cdots \to v_r=t$$
of $G$ is {\em admissible} if $s<t$ and, for $k=1,2,\dots,r-1$, one has either $v_k<s$ or $v_k >t$.
The vertices $s$ and $t$ are called the {\em ends} of $P$
and the vertices $v_1,\dots,v_{r-1}$ are called the {\em inner vertices} of $P$.

For an admissible path $P: s = i_0 \to i_1 \to \cdots \to i_r=t$,
we define the monomial
$$m_P = \left( \prod_{v_k <s} y_{v_k} \right) \left( \prod_{v_k >t} x_{v_k} \right) x_s y_t.$$
Let $\mathcal P(G)$ be the set of all admissible paths of $G$,
and let $>_\lex$ be the lexicographic order induced by $x_1> \cdots>x_n > y_1 > \cdots >y_n$.
For an ideal $I \subset S$,
let $\init_{>_\lex}(I)$ be the initial ideal of $I$ w.r.t.\ $>_\lex$.
The following result is due to Herzog et.al.\ \cite[Theorem 2.1]{HHHKR} and Ohtani \cite[Theorem 3.2]{Oh}.

\begin{lemma}
\label{3.2}
$\init_{>_\lex}(J_G)= (m_P: P \in \mathcal P (G))$.
\end{lemma}

Note that our definition of the admissibility is different to that in \cite{HHHKR}.
In particular, the generators in Lemma \ref{3.2} may not be minimal.

The next property is a key lemma to prove the main result.

\begin{lemma}
\label{3.3}
Let $P: s=v_0 \to \cdots \to v_r=t$ be an admissible path and $1 \leq k \leq r-1$.
\begin{itemize}
\item[(i)]
If $v_k<s$ then there is an $\ell >k$ such that
$P': v_k \to v_{k+1} \to \cdots \to v_\ell$
is an admissible path of $G$ and $m_{P'}$ divides $x_{v_k} m_P$.
\item[(ii)]
If $v_k>t$ then there is an $\ell <k$ such that
$P': v_\ell \to v_{\ell+1} \to \cdots \to v_k$
is an admissible path of $G$ and $m_{P'}$ divides $y_{v_k} m_P$.
\end{itemize}
\end{lemma}

\begin{proof}
We prove (i) (the proof for (ii) is similar).
Let $\ell>k$ be the smallest integer satisfying $ i_k < i_\ell \leq t$.
Then the path 
$P': v_k \to v_{k+1} \to \cdots \to v_\ell$
satisfies the desired condition.
\end{proof}

We call a path $P'$ satisfying condition (i) or (ii) in Lemma \ref{3.3} an {\em wedge} of $P$ at $v_k$.

From now on,
we fix an ordering
$$P_1,P_2,\dots,P_g$$
of the admissible paths of $G$,
where $g=\# \mathcal P (G)$,
such that if the length of $P_i$ is smaller than that of $P_j$ then $i<j$.
To simplify the notation, we write
$$m_k=m_{P_k}$$
for $k=1,2,\dots,g$.
Then $\init_{>_\lex}(J_G)=(m_1,\dots,m_g)$.
By the choice of the ordering, if $P_i$ is an wedge of $P_j$ then $i<j$.
This fact immediately implies the following property.

\begin{lemma}
\label{3.4}
Let $1 < j \leq g$ and let $s$ and $t$ be the ends of $P_j$ with $s<t$.
For any inner vertex $v$ of $P_j$,
one has $x_v \in (m_1,\dots,m_{j-1}):m_j$ if $v <s$
and $y_v \in (m_1,\dots,m_{j-1}):m_j$ if $v >t$.
\end{lemma}

For a monomial $w \in S$, let
$$\mult(w)= \{ k \in[n]: x_k y_k \mbox{ divides }w\}.$$
Note that, for a squarefree monomial $w \in S$,
one has $ \deg w \leq n+ \# \mult(w)$.
Since the regularity does not decrease under taking initial ideals (see e.g., \cite[Theorem 22.9]{P}),
the next statement proves the remaining part of Theorem \ref{1.1}.

\begin{proposition}
\label{3.5}
For any monomial $w \in S$ and an integer $p>0$,
one has
$$\beta_{p,w}\left( S/\init_{>_\lex}(J_G)\right) =0
\ \ \mbox{ if } \# \mult(w) \geq p.$$
In particular,
$\reg( \init_{>_\lex}(J_G)) \leq n$.
\end{proposition}

\begin{proof}
The second statement follows from the first statement together with the fact that the multigraded Betti numbers
of a squarefree ideal is concentrated in squarefree degrees.
Thus we prove the first statement.

We first introduce notations.
Let $\mathcal M=\{m_1,m_2,\dots,m_g\}$.
We say that a subset $F=\{m_{i_1},m_{i_2},\dots,m_{i_k}\} \subset \mathcal M$, where $i_1< \cdots <i_k$,
is a {\em Lyubeznik subset} of $\mathcal M$ (of size $k$) if, for $j=1,2,\dots,k$,
any monomial $m_\ell$ with $\ell < i_j$ does not divide $\lcm(m_{i_j},m_{i_{j+1}}, \dots, m_{i_k})$.
We prove the assertion by the following two claims.
\medskip

\noindent
\textbf{Claim 1.}
Let $F=\{m_{i_1},\dots,m_{i_k}\}$, where $i_1< \cdots <i_k$, be a Lyubeznik subset of $\mathcal M$.
Then
\begin{itemize}
\item[(i)] $\mult(\lcm(F))$ contains no inner vertices of $P_{i_1}$.
\item[(ii)] if $\mult(\lcm(F))$ contains no inner vertices of $P_{i_j}$ for $j=2,3,\dots,k$
then $\# \mult(\lcm(F)) \leq k-1$.
\end{itemize}
\medskip

\noindent
\textbf{Claim 2.}
Let $F=\{m_{i_1},\dots,m_{i_k}\}$, where $i_1< \cdots <i_k$, be a Lyubeznik subset of $\mathcal M$
and $w$ a monomial of $S$.
Let $p>0$ be an integer.
Suppose
\begin{itemize}
\item[(a)] $\beta_{p,w}(S/((m_1,\dots,m_{i_1-1}):m_{i_1} \cdots m_{i_k}) ) \ne 0$, and
\item[(b)] $\mult(w \cdot  \lcm(F))$ contains no inner vertices of $P_{i_\delta}$ for $\delta=2,3,\dots,k$.
\end{itemize}
Then there is a Lyubeznik subset $\widetilde F=\{m_{j_1},\dots,m_{j_\ell}\}$, where $j_1 < \cdots < j_\ell$,
of $\mathcal M$ and a monomial $\widetilde w$ such that
\begin{itemize}
\item[(a')] $\beta_{p-1,\widetilde w}(S/((m_1,\dots,m_{j_1-1}):m_{j_1} \cdots m_{j_\ell})) \ne 0$,
\item[(b')] $\mult(\widetilde w \cdot  \lcm( \widetilde F))$ contains no inner vertices of $P_{j_\delta}$ for $\delta=2,3,\dots,\ell$, and
\item[(c')] $\# \mult(\widetilde w \cdot  \lcm (\widetilde F))- \# \widetilde F =\# \mult(w\cdot   \lcm(F)) - \#F -1$.
\end{itemize}
\medskip

We first show that these claims prove the desired statement.
Let $u \in S$ be a monomial such that $\beta_{p,u}(S/\init_{>_\lex}(J_G)) \ne 0$ with $p>0$.
We show that there is a Lyubeznik subset $F$
such that 
\begin{align}
\label{c0}
\# \mult(u)=\# \mult (\lcm(F)) -\#F + p
\end{align}
and $F$ satisfies the assumption of Claim 1(ii).
Note that this proves the desired statement by Claim 1(ii).

Recall $\init_{>_\lex}(J_G)=(m_1,\dots,m_g)$.
By Lemma \ref{3.1},
there is a Lyubeznik subset $\{m_j\}$ of size $1$ such that $\beta_{p-1,u/m_j}(S/((m_1,\dots,m_{j-1}):m_j)) \ne 0$.
If $p=1$ then $u=m_j$ and the set $\{m_j\}$ has the desired property \eqref{c0}.
Suppose $p>1$.
Then the pair of the Lyubeznik set $\{m_j\}$ and a monomial $u/m_j$ satisfies the assumption (a) and (b) of Claim 2.
Thus, by applying Claim 2 repeatedly,
one obtains a Lyubeznik subset $F=\{m_{i_1},\dots,m_{i_k}\}$ and a monomial $w$ such that
\begin{itemize}
\item $\beta_{0,w}(S/((m_1,\dots,m_{i_1-1}):m_{i_1} \cdots m_{i_k})) \ne 0$, and
\item $\# \mult(w \cdot \lcm(F))-\#F = \# \mult (u) -p$.
\end{itemize}
The first condition says $w=x^{\mathbf 0} y^{\mathbf 0}$,
where $\mathbf 0=(0,\dots,0)$,
and the second condition proves that $F$ satisfies the desired property \eqref{c0}.

In the rest, we prove Claims 1 and 2.
\begin{proof}[Proof of Claim 1]
(i) Suppose to the contrary that there is an inner vertex $v$ of $P_{i_1}$
which belongs to $\mult(\lcm(F))$.
Let $P_j$ be a wedge of $P_{i_1}$ at $v$.
Then $j<i_1$ and $m_j$ divides $\lcm(m_{i_1},\dots,m_{i_k})$ by Lemma \ref{3.3}.
This contradicts the definition of Lyubeznik sets.

(ii)
Let $s_1,t_1,s_2,t_2,\dots,s_k,t_k$ be the ends of $P_{i_1},\dots,P_{i_k}$,
where $s_j<t_j$ for all $j$.
By (i) and the assumption, $\mult(\lcm(F))$ contains no inner vertices of $P_{i_j}$ for all $j$.
Hence
$$\# \mult(\lcm(F)) \leq \# \mult(x_{s_1}y_{t_1} x_{s_2}y_{t_2} \cdots x_{s_k}y_{t_k} )\leq k-1,$$
where the last inequality follows from $s_1<t_1,\dots,s_k<t_k$.
\end{proof}

\noindent
\textit{Proof of Claim 2.}
We consider two cases.

\textit{Case 1:}
Suppose that $\mult(w \cdot \lcm(F))$ contains an inner vertex $v$ of $P_{i_1}$.
Consider the case that $x_v$ divides $m_{i_1}$
(the case that $y_v$ divides $m_{i_1}$ is similar).
Since $y_v$ does not divide $\lcm(F)$ by Claim 1(i),
$y_v$ divides $w$.
Then, as $y_v \in (m_1,\dots,m_{i_1-1}):m_{i_1} \cdots m_{i_k}$ by Lemma \ref{3.4},
we have $\beta_{p,w}(S/((m_1,\dots,m_{i_1-1}):m_{i_1} \cdots m_{i_k})) \ne 0$
if and only if $\beta_{p-1,w/y_v}(S/((m_1,\dots,m_{i_1-1}):m_{i_1} \cdots m_{i_k})) \ne 0$.
Then the pair of the set $\widetilde F=F$ and the monomial $\widetilde w = w /y_v$ satisfies (a'), (b') and (c')
as desired.

\textit{Case 2:}
Suppose that $\mult(w \cdot \lcm(F))$ contains no inner vertices of $P_{i_1}$.
For $j=1,2,\dots,i_1-1$,
let
$$\overline{m}_j = \frac {m_j} {\mathrm{gcd}(m_j,m_{i_1}\cdots m_{i_k})}.$$
Then we have
$$(\overline{ m}_1,\dots,\overline{ m}_{i_1-1}) = (m_1,\dots,m_{i_1-1}):m_{i_1} \cdots m_{i_k}.$$
By Lemma \ref{3.1} and (a),
there is an $1 \leq i_0 < i_1$ such that $\overline{m}_{i_0} \not \in (\overline{m}_1,\dots,\overline{m}_{i_0-1})$
and
\begin{align}
\label{c1}
\beta_{p-1,w/\overline{m}_{i_0}} \left( S/ \big( (\overline{m}_1,\dots,\overline{m}_{i_0-1}): \overline{m}_{i_0}\big) \right) \ne 0.
\end{align}
Let $\widetilde w = w /\overline{m}_{i_0}$ and $\widetilde F=\{m_{i_0},m_{i_1}\dots,m_{i_k}\}$.
Since, for $\ell <i_0$, $\overline{m}_\ell$ divides $\overline{m}_{i_0}$ if and only if $m_\ell$ divides $\lcm(m_{i_0}, m_{i_1}, \dots, m_{i_k})$,
$\widetilde F$ is a Lyubeznik subset.
Also,
since
\begin{align*}
(\overline{m}_1,\dots,\overline{m}_{i_0-1}):\overline{m}_{i_0}  = (m_1,\dots,m_{i_0-1}):m_{i_0}m_{i_1} \cdots m_{i_k},
\end{align*}
\eqref{c1} and	 the fact $w \cdot \lcm(F)= \widetilde w \cdot \lcm (\widetilde F)$
say that the pair $\widetilde F$ and $\widetilde w$ satisfies (a'), (b') and (c') as desired.
\end{proof}

\begin{remark}
Although we use conditions that appear in Lyubeznik resolutions \cite{Ly},
Lyubeznik resolutions themselves seem not to prove Proposition \ref{3.5}.
\end{remark}

\begin{example}
Both inequalities in Theorem \ref{1.1} could be strict.
Indeed,
the regularity of the binomial edge ideal of the following graph is $6$.
However, the graph has $7$ vertices and the length of its longest induced path is $4$.
\bigskip
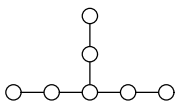
\begin{center}
\unitlength 0.1in
\begin{picture}(  8.8000,  4.8000)(  9.6000,-12.4000)
%
\special{pn 8}%
\special{pa 1400 1200}%
\special{pa 1800 1200}%
\special{fp}%
\special{pa 1400 1200}%
\special{pa 1400 800}%
\special{fp}%
\special{pa 1400 1200}%
\special{pa 1000 1200}%
\special{fp}%
%
\special{pn 8}%
\special{sh 0}%
\special{ar 1000 1200 40 40  0.0000000 6.2831853}%
%
\special{pn 8}%
\special{sh 0}%
\special{ar 1200 1200 40 40  0.0000000 6.2831853}%
%
\special{pn 8}%
\special{sh 0}%
\special{ar 1400 1200 40 40  0.0000000 6.2831853}%
%
\special{pn 8}%
\special{sh 0}%
\special{ar 1400 1000 40 40  0.0000000 6.2831853}%
%
\special{pn 8}%
\special{sh 0}%
\special{ar 1400 800 40 40  0.0000000 6.2831853}%
%
\special{pn 8}%
\special{sh 0}%
\special{ar 1600 1200 40 40  0.0000000 6.2831853}%
%
\special{pn 8}%
\special{sh 0}%
\special{ar 1800 1200 40 40  0.0000000 6.2831853}%
\end{picture}%
\end{center}
\end{example}

\begin{remark}
A similar bound holds for the depth of $S/J_G$.
Let $K_n$ be the complete graph on $[n]$.
If $G$ is a connected graph on $[n]$
then $J_{K_n}$ is an associated prime of $S/J_G$ by \cite[Corollary 3.9]{HHHKR}
and $\dim S/J_{K_n}=n+1$.
This fact implies $\mathrm{depth} (S/J_G) \leq n+1$ (see \cite[Propositon 1.2.13]{BH}).
\end{remark}

We end this note with the following conjecture.

\begin{conjecture}
\label{conj}
Let $G$ be a graph on $[n]$.
If $\reg(J_G)=n$ then $G$ is a path of length $n$.
\end{conjecture}

We verify Conjecture \ref{conj}
for graphs with at most $9$ vertices in characteristic $0$ and $2$ by using Macaulay2 \cite{GS}.
For this computation, we use the list of graphs with at most $9$ vertices in \cite{Ma}.


\end{document}